\documentclass[a4paper,10pt]{amsart}  

\usepackage{dacxjo} 

\title{Non-normal purely log terminal centres in characteristic $p \geq 3$}
\author{
Fabio Bernasconi}

\subjclass[2010]{14E30, 14J17}
\keywords{Fano varieties, vanishing theorems, MMP singularities, positive characteristic}
\newcommand{\Addresses}{{
  \bigskip
  \footnotesize

  \textsc{Department of Mathematics, Imperial College, London, 180 Queen’s Gate,
  London SW7 2AZ, UK}\par\nopagebreak
  \textit{E-mail address:} \texttt{f.bernasconi15@imperial.ac.uk}
  \medskip
}}

\begin{document}
\begin{abstract} 
In this note we show, building on a recent work of Totaro \cite{Tot17}, that for every prime number $p \geq 3$ there exists a purely log terminal pair $(Z,S)$ of dimension $2p+2$ whose plt centre $S$ is not normal. 
\end{abstract}
\maketitle
\section{Introduction}
The minimal model program (MMP for short) has become a fundamental tool in the study of the geometry of algebraic varieties. While a large part of the program is known to hold in characteristic zero (see the seminal work \cite{BCHM10}), we have very few results for varieties over fields of positive characteristic in arbitrary dimension. In particular, the existence of flips is still a major open problem in characteristic $p>0$. The main difficulty lies in the failure of Kodaira type vanishing theorems, which makes the study of  singularities in positive characteristic far more complicated than in characteristic zero. For example we know that there exist Kawamata log terminal (klt) and even terminal singularities which are not Cohen-Macaulay in positive characteristic (see \cite{CT16}, \cite{Kov17}, \cite{Ber17}, \cite{Tot17} and \cite{Yas17}). \\
\indent Another important property of singularities in characteristic zero due to Kawamata-Viehweg vanishing is the normality of the centre $S$ of a purely log terminal (plt) pair $(X,S+B)$ (see \cite[Proposition 5.51]{KM98}), which is a crucial ingredient in the proof of the existence of (pl-)flips in characteristic zero (see \cite{HM07} and \cite{BCHM10}). Although Kodaira vanishing is no longer valid in positive characteristic, in \cite[Theorem 3.1.1 and Proposition 4.1]{HX15} the authors show the normality of plt centres for threefolds over an algebraically closed field of characteristic $p > 5$ and, using tools from the theory of $F$-singularities, they succeed in proving the existence of (pl-)flips (\cite[Theorem 4.12]{HX15}). This result was the starting point of a series of works (see \cite{CTX15}, \cite{Bir16}, \cite{BW17} and \cite{HNT17}) which established a large part of the log MMP for threefolds over fields of characteristic $p > 5$. \\
\indent One might thus be led to conjecture that plt centres are normal also in positive characteristic. Unfortunately, this is not the case: in \cite[Theorem 1]{CT16}, the authors construct an example of a plt threefold with non-normal centre in characteristic two. Inspired by their work, we use a recent series of examples of Fano varieties violating Kodaira vanishing built by Totaro (see \cite{Tot17}) to construct new examples of non-normal plt centres for every prime $p \geq 3$:
\begin{thm}\label{main}
Let $k$ be an algebraically closed field of characteristic $p \geq 3$. Then there exists a log pair $(Z,S)$ such that 
\begin{enumerate}
\item $Z$ is an affine variety over $k$ with terminal singularities of dimension $2p+2$ and $S$ is a prime divisor,
\item $(Z,S)$ is a purely log terminal pair with $K_Z+S$ Cartier,
\item $S$ is not normal.
\end{enumerate}
\end{thm}
As in \cite[Theorem 1.2]{CT16}, we deduce that the lifting lemma of Hacon-M$^\text{c}$Kernan for pluri-log-canonical forms from plt centres in characteristic zero (see \cite[Theorem 5.4.21]{HM07}) fails in general over fields of positive characteristic:
\begin{corollary} \label{lift}
With the same notation as in Theorem \ref{main}, there exists a projective birational morphism $f \colon (Y,S^Y) \rightarrow (Z,S)$ such that
\begin{enumerate} 
\item $(Y,S^Y)$ is log smooth and $S^Y$ is a prime divisor,
\item $-\Exc(f)$ is an ample divisor,
\item $K_Y+S^Y$ is semiample and big, and
\item for every $m \geq 0$, the restriction map
\[H^0(Y, \mathcal{O}_Y(m(K_Y+S^Y))) \rightarrow H^0(S^Y, \mathcal{O}_Y(mK_{S^Y})),  \]
is not surjective.
\end{enumerate}
\end{corollary}
In the last section we give examples of terminal Fano varieties with non-vanishing intermediate cohomology group by taking the projective cone over Totaro's example.
\begin{thm}\label{terminalFanononvanishingcohomology}
Let $k$ be an algebraically closed field of characteristic $p \geq 3$. Then there exists a Fano variety $Z$ with terminal singularities of dimension $2p+2$ over $k$ such that
\[H^2(Z, \mathcal{O}_Z) \neq 0. \]
\end{thm}
\subsection*{Acknowledgements}
I would like to express my gratitude to my advisor Paolo Cascini for his constant support and to Mirko Mauri, Hiromu Tanaka, Burt Totaro and Jakub Witaszek
for reading an earlier version of this note and for their insightful comments. This work was supported by the Engineering and Physical Sciences Research Council [EP/L015234/1], The EPSRC Centre for Doctoral Training in Geometry and Number Theory (The London School of Geometry and Number Theory), University College London and Imperial College, London.
\section{Preliminaries}
\subsection{Notation} Throughout this note, $k$ denotes an algebraically closed field. By \emph{variety} we mean an integral scheme which is separated and of finite type over $k$.  If $X$ is a normal variety, we denote by $K_X$ the canonical divisor class. We say that $(X,\Delta)$ is a \emph{log pair} if $X$ is a normal variety, $\Delta$ is an effective $\Q$-divisor and $K_X+\Delta$ is $\Q$-Cartier. We say it is \emph{log smooth} if $X$ is smooth and $\Supp(\Delta)$ is a snc divisor. We refer to \cite{KM98} and \cite{Kol13} for the definition of the singularities appearing in the MMP (e.g. \emph{terminal, klt, plt, dlt}). We say $f \colon Y \rightarrow X$ is a \emph{log resolution} if $f$ is a birational morphism, $Y$ is regular, $\Exc(f)$ has pure codimensione one and the pair $(Y, \Supp(f^{-1}\Delta + \Exc(f)))$ is log smooth. We denote by $\rho(X):= \rho(X/\Spec k)$ the \emph{Picard number} of $X$.
\subsection{Fano varieties violating Kodaira vanishing in positive characteristic}
While counterexamples to Kodaira vanishing for general projective varieties in positive characteristic are quite abundant in literature (see for example \cite{Ray78}), until the recent work of Totaro (see \cite{Tot17}) only few examples of smooth Fano varieties violating the Kodaira vanishing theorem were known (see \cite{LR97} and \cite{Kov17}). We recall the properties of Totaro's example we need to prove Theorem \ref{main}:
\begin{thm}[{\cite[Theorem 2.1]{Tot17}}]\label{riassunto}
Let $k$ be an algebraically closed field of characteristic $p \geq 3$. Then there exists a smooth Fano variety $X$ over $k$ of dimension $2p+1$ with a very ample Cartier divisor $A$ such that
\begin{enumerate}
\item $\rho(X)=2$,
\item $-K_X=2A$ and
\item $H^1(X, \mathcal{O}_X(A)) \neq 0$.
\end{enumerate}
\end{thm}
This theorem has already been a prolific source of examples of pathologies in positive characteristic: by taking the cone over $X$, Totaro shows that for every $p \geq 3$ there exists a terminal not Cohen-Macaulay singularity in dimension $2p+2$ (\cite[Corollary 2.2]{Tot17}) and in \cite{AZ17} the authors construct non-liftable Calabi-Yau varieties by considering anticanonical sections of $X$ and double covers along a general member of $|-2K_X|$.
\subsection{Affine cones}
For the theory of cones over algebraic varieties and the study of their singularities we refer to \cite[Chapter 3]{Kol13}. Here we recall only what we need to prove our main result. \\
\indent Let $(X,\Delta)$ be a log pair over $k$ and let $L$ be an ample line bundle on $X$. We denote by $C_a(X,L):=\Spec_X \sum_{m \geq 0} H^0 (X,L^{\otimes m})$ the cone over $X$ induced by $L$ and by $\Delta_{C_a(X,L)}$ the induced $\Q$-divisor on $C_a(X,L)$. Over $X$ we consider the $\mathbb{A} ^1-$bundle $\pi \colon BC_a(X,L) := \Spec_X \sum_{m \geq 0} L^{\otimes m} \rightarrow X$ and we have the following diagram:
\begin{equation*}
\xymatrix{
	 BC_a(X,L) \ar[r]^{\qquad \pi} \ar[d]_{f} & X  \\
	 C_a(X,L) &,
}
\end{equation*}
where $f$ is the birational morphism contracting the section $X^{-}$ of $\pi$ with anti-ample normal bundle $\mathcal{N}_{X^{-}/BC_a(X,L)} \simeq L^{\vee}$. The following describes the divisor class group of the cone and the condition under which its canonical class is $\Q$-Cartier:
\begin{prop} [{\cite[Proposition 3.14]{Kol13}}] \label{Pic}
With the same notation as above, we have
\begin{enumerate}
\item $\Cl(C_a(X,L))=\Cl(X)/ \langle L \rangle$,
\item $\Pic(C_a(X,L))=0$, and
\item $m(K_{C_a(X,L)}+\Delta_{C_a(X,L)})$ is Cartier if and only if there exists $r \in \Q$ such that $\mathcal{O}_X(m(K_X+\Delta)) \simeq L^{rm}$. \label{canonicalclass}
\end{enumerate}
\end{prop}
From the point of view of the singularities of the MMP we have the following
\begin{prop}[{\cite[Lemma 3.1]{Kol13}}]\label{plt}
With the same notation as above, let us assume that $K_X + \Delta \sim_{\Q} rL$. Then we have
\begin{equation} \label{discrep}
K_{BC_a(X,L)}+\pi^*\Delta+(1+r)X^-=f^*(K_{C_a(X,L)}+\Delta_{C_a(X,L)});
\end{equation} 
and the pair $(C_a(X,L), \Delta_{C_a(X,L)})$ is terminal (resp. dlt) if and only if the pair $(X,\Delta)$ is terminal (resp. dlt) and $r < -1$ (resp. $r <0$).
\end{prop}
\subsection{Projective cones}
Let $X$ be a normal variety over $k$ and let $L$ be an ample line bundle on $X$. We define the projective cone of $X$ with respect to $L$ as in \cite[Section 3.8]{Kol13}:
\[C_p(X,L)= \Proj_k \sum_{m \geq 0}(\sum_r^m H^0(X,L^r)x_{n+1}^{m-r}). \]
It contains as an open dense subset $C_a(X,L)$ and it admits a partial resolution
\[\pi \colon \mathbb{P}_X(\mathcal{O}_X \oplus L) \rightarrow C_p(X,L). \]
We show how to compute the cohomology of the structure sheaf of $C_p(X,L)$ in terms of the cohomology groups of $L$ on $X$:
\begin{prop}\label{cohomologyprojectivecone} For $i \geq 2$,
\[H^i(C_p(X,L), \mathcal{O}_{C_p(X,L)}) \simeq \sum _{m >0} H^{i-1}(X, L^{m}). \]
\end{prop}
\begin{proof}
We denote by $v$ the vertex of $C_a(X,L)\subset C_p(X,L)$ and we consider the natural inclusion 
\[V := C_a(X,L) \setminus v= \Spec_X \sum_{m \in \Z} L^m \subset  \Spec_X \sum_{m \leq 0} L^m= C_p(X,L) \setminus v =:U. \]
Considering the long exact sequences in local cohomology (see \cite[Chapter III, Ex 2.3]{Ha77}) we have the following natural commutative diagram
\begin{equation*}
\xymatrix{
	 \dots \ar[r] & H^{i-1}(U, \mathcal{O}_U) \ar[r]^{\delta_i \qquad \,} \ar[d]^{i^*} & H^i_v(C_p(X,L),\mathcal{O}_{C_p(X,L)}) \ar[r] \ar[d]^{\simeq} & H^i(C_p(X,L),\mathcal{O}_{C_p(X,L)})\ar[r] \ar[d]^{i^*} & \dots \\
	 \dots \ar[r] & H^{i-1}(V, \mathcal{O}_V) \ar[r]^{\eta_i \qquad \,} & H^i_v(C_a(X,L),\mathcal{O}_{C_a(X,L)}) \ar[r] & H^i(C_a(X,L),\mathcal{O}_{C_a(X,L)}) \ar[r] & \dots,
}
\end{equation*}
where by $i$ we mean the natural inclusion maps.
It is easy to see that the diagram
\begin{equation*}
\xymatrix@1{
	 H^{i}(U, \mathcal{O}_U) \ar[r]^{i^*} \ar[d]^\simeq & H^{i}(V,\mathcal{O}_V) \ar[d]^\simeq \\
	 \sum_{i \leq 0} H^{i-1}(X,L^m) \ar[r] & \sum_{m \in \Z} H^{i-1}(X, L^m)
}
\end{equation*}
commutes, where the bottom arrow is the natural inclusion, thus showing that $i^*$ is injective. Since for $i \geq 2$ the maps $\eta_i$ are isomorphisms (due to the vanishing $H^{i-1}(C_a(X,L), \mathcal{O}_{C_a(X,L)})=0$), we conclude that $\delta_i$ are injective and thus we have the following commutative diagram of short exact sequences:
\begin{equation*}
\xymatrix@1{
	 0 \ar[r]\ar[r] & H^{i-1}(U, \mathcal{O}_U) \ar[r]^{\delta_i \qquad \,} \ar[d]^{\simeq} & H^i_v(C_p(X,L),\mathcal{O}_{C_p(X,L)}) \ar[r] \ar[d]^{\simeq} & H^i(C_p(X,L),\mathcal{O}_{C_p(X,L)}) \ar[r] \ar[d]^{\simeq} & 0 \\
	 0 \ar[r] & \sum_{m \leq 0} H^{i-1}(X,L^m)\ar[r] & \sum_{m \in \Z} H^{i-1}(X,L^m) \ar[r] & H^i(C_p(X,L),\mathcal{O}_{C_p(X,L)}) \ar[r] & 0,
}
\end{equation*}
which concludes the proof.
\end{proof}
\section{Non-normal plt centres}
To prove the main result we consider a cone over the Fano variety $X$ constructed by Totaro and we show that the  prime divisor induced on the cone by a smooth section $E \in |A|$ is not normal.
\begin{proof}[Proof of Theorem \ref{main}]
Let us fix an algebraically closed field $k$ of characteristic $p \geq 3$. Let us consider the smooth Fano variety $X$ with the ample divisor $A$ of Theorem \ref{riassunto}. Since $A$ is very ample, by Bertini theorem there exists a smooth divisor $E \in |A|$. We define the pair $(Z,S):=(C_a(X,\mathcal{O}_{X}(A)), E_{C_a(X,\mathcal{O}_{X}(A))})$. Since $(X,E)$ is log smooth and $K_X+E \sim -A$ we conclude by Proposition \ref{Pic} that $K_Z+S$ is Cartier and by Proposition \ref{plt} that the pair $(Z,S)$ is dlt. Since $S$ is the only irreducible component in the boundary, the pair $(Z,S)$ is actually plt.  We denote by $(Y, S^Y)$ the pair given by $(BC_a(X, \mathcal{O}_X(A)), f_*^{-1}S)$. \\
\indent We check that $S$ is not normal. Let us note that
\[S = \Spec_k \sum_{m \geq 0} \text{Im}(H^0(X, \mathcal{O}_X(mA)) \rightarrow H^0(E, \mathcal{O}_E(mA)). \]
Let us consider the short exact sequence
\[0 \rightarrow \mathcal{O}_X(-E) \rightarrow \mathcal{O}_X \rightarrow \mathcal{O}_E \rightarrow 0, \]
and tensor it with $nA$, where $n$ is a positive integer. Taking the long exact sequence in cohomology we have
\[H^0(X, \mathcal{O}_X(nA)) \rightarrow H^0(E, \mathcal{O}_E(nA)) \rightarrow H^1(X, \mathcal{O}_X((n-1)A)) \rightarrow H^1(X, \mathcal{O}_X(nA)).\]
Since $A$ is ample and $H^1(X,\mathcal{O}(A)) \neq 0$, we can consider, by Serre vanishing, the largest $n \geq 2$ such that $H^1(X, \mathcal{O}_X((n-1)A)) \neq 0 $ and $ H^1(X, \mathcal{O}_X(nA))=0$. Thus the morphism 
\[H^0(X, \mathcal{O}_X(nA)) \rightarrow H^0(E, \mathcal{O}(nA)) \]
is not surjective and therefore the natural morphism
\[\nu \colon S^{\nu}:=C_a(E,L|_E) \rightarrow S, \]
is not an isomorphism. The morphism $\nu$ is finite and it is birational since for $m>n$ we have $H^0(E, \mathcal{O}_E (mA))= \text{Im} (H^0(X, \mathcal{O}_X(mA)) \rightarrow H^0(E, \mathcal{O}_E(mA)).$
Thus we conclude that the variety $S$ is not normal and that $\nu$ is the normalisation morphism, since $S^\nu$ is normal.
\end{proof}
\begin{rem}
We note that, since $\rho(X)=2$, the affine variety $Z$ is not $\Q$-factorial by Proposition \ref{Pic}.
\end{rem}
\begin{rem}
We point out that $S$ is regular in codimension one and thus by Serre's criterion we deduce that $S$ does not satisfy the $S2$ condition. In general, plt centres satisfy the $R1$ condition also in positive characteristic as explained in \cite[Lemma 2.5]{GNT15}.
\end{rem}
\begin{rem}
The normalisation morphism $\nu \colon S^{\nu} \rightarrow S$ is a universal homeomorphism since the morphism $S^Y \rightarrow S$ has connected fibers. In \cite[Theorem 3.15]{GNT15}, the authors prove, assuming the existence of pl-flips, that this is always the case for plt centres on threefolds in any characteristic.
\end{rem}
\begin{proof}[Proof of Corollary \ref{lift}] 
We verify that properties $(1)-(3)$ hold. The first property is immediate by construction of $(Y,S^Y)$ as an $\mathbb{A}^1$-bundle over a log smooth pair. To show property (2) we note that the divisor $-X^-$ is $f$-ample over the affine variety $Z$, thus it is ample. We have, by formula (\ref{discrep}) of Proposition \ref{plt}, that
\[K_Y+S^Y = f^*(K_Z+S),\]
and since $K_Z+S$ is ample on $Z$ we conclude property $(3)$  holds. \\
\indent We now show property (4). Since $S$ is not normal we have that the morphism
\begin{equation} \label{surj}
f_* \mathcal{O}_Y \rightarrow f_* \mathcal{O}_{S^Y}
\end{equation}
is not surjective. Indeed, since $S^Y \rightarrow S^{\nu} $ is a surjective birational projective morphism between normal varieties, we have by Zariski's main theorem the following commutative diagram:
\begin{equation*}
\xymatrix{
	 \mathcal{O}_Z \ar@{->>}[r] \ar[dd]_{\simeq} & \mathcal{O}_S \ar[dd] \ar[dr] \\
	 &  & \nu_*\mathcal{O}_{S^\nu} \ar[dl]_{\simeq}  \\
	 f_*\mathcal{O}_Y \ar[r] & f_*\mathcal{O}_{S^Y} &.
}
\end{equation*}
Since $\mathcal{O}_S \rightarrow \nu_*\mathcal{O}_{S^\nu}$ is not surjective, we conclude that the bottom arrow is not surjective. \\
\indent By Proposition \ref{Pic} we have that $K_Z+S \sim 0$ and thus, using the projection formula, we have
\[f_* \mathcal{O}_Y (m(K_Y+S^Y)) = f_* (f^*\mathcal{O}_Z(m(K_Z+S)))=  f_* \mathcal{O}_Y,\]
and in the same way
\[ f_* \mathcal{O}_{S^Y} (mK_{S^Y})= f_* \mathcal{O}_{S^Y}. \]
Since $Z$ is affine, we have therefore that the morphism in (\ref{surj}) is not surjective if and only if
\[H^0(Y, \mathcal{O}_Y(m(K_Y+S^Y))) \rightarrow H^0(S^Y, \mathcal{O}_{S^Y}(mK_{S^Y})), \]
is not surjective for any integer $m > 0$, thus concluding the proof.
\end{proof}
\section{Terminal Fano varieties with $H^2(\mathcal{O}_Z) \neq 0$}
In this section we prove Theorem \ref{terminalFanononvanishingcohomology}. Since we need to understand the positivity properties of the canonical divisor of a projective bundle we recall the Euler sequence. Let $E$ be a vector bundle of rank $r+1$ on a variety $X$ and let $\pi \colon \mathbb{P}_X(E) \rightarrow X$ be the associated projective bundle of hyperplanes, then we have the following short exact sequence:
\[ 0 \rightarrow \Omega^1_{\mathbb{P}(E)/X} \rightarrow \mathcal{O}_{\mathbb{P}(E)}(-1) \otimes \pi^* E \rightarrow \mathcal{O}_{\mathbb{P}(E)} \rightarrow 0, \]
which shows $K_{\mathbb{P}(E)}= \mathcal{O}_{\mathbb{P}(E)}(-r-1) \otimes \pi^*(K_X \otimes \det E)$.
\begin{proof}[Proof of Theorem \ref{terminalFanononvanishingcohomology}]
Let us fix an algebraically closed field $k$ of characteristic $p \geq 3$ and let us consider the Fano variety $X$ with the ample divisor $A$ of Theorem \ref{riassunto}. We define the projective variety $Z:= C_p(X, \mathcal{O}_X(A))$. Since $H^1(X, \mathcal{O}_X(A)) \neq 0$, we conclude $H^2(Z, \mathcal{O}_Z) \neq 0$ by Proposition \ref{cohomologyprojectivecone}. The variety $Z$ has terminal singularities since the only singular point is the vertex of the cone $C_a(X,\mathcal{O}_X(A))$. \\
\indent We are only left to prove that $Z$ is a Fano variety. For this it is sufficient to check that the projective bundle $\pi \colon Y:=\mathbb{P}_X(\mathcal{O}_X \oplus \mathcal{O}_X(-A)) \rightarrow X$ is Fano. On $Y$ there is a unique negative section $X^-$ such that $\mathcal{O}_{Y}(X^-)|_{X^-}\simeq \mathcal{O}_X(-A)$. There exists also a positive section $X^+$ such that $\mathcal{O}_{Y}(X^+)|_{X^+}\simeq \mathcal{O}_{X}(A)$ (which implies that $X^+$ is a big and nef divisor on $Y$) and such that  $X^+ \sim X^-+\pi^*A$. We note that
\[\mathcal{O}_Y(1) = \mathcal{O}_Y(X^-),\]
and by the relative Euler sequence we have
\[K_{Y} =\mathcal{O}_Y(-2) \otimes \pi^{*}(K_X-A). \]
Thus we have that the anticanonical class
\begin{equation} \label{anticanonical}
-K_Y=2X^- + \pi^*(3A)=2X^+ + \pi^*A,
\end{equation}
is a big and nef divisor. To conclude that $-K_Y$ is ample, we show that the null locus $\text{Null}(-K_Y)$ (see \cite[Definition 10.3.4]{Laz04-2}) is empty. By equation (\ref{anticanonical}), the null locus must be contained in $X^-$, but since $-K_Y|_{X^-}=A$ is ample we conclude it must be empty.
\end{proof}
\section{Open questions}
A question that naturally arises from Theorem \ref{main} and from the result of \cite{HX15} for threefolds in characteristic $p>5$ is the following: 
\begin{question}
Fixed a positive integer $n$, does there exists a number $p_0(n)$ such that $n$-dimensional plt centres in characteristic $p \geq p_0(n)$ are normal?
\end{question}
A guiding philosophy in the study of MMP singularities in large characteristic is that their good behaviour should correspond to Kodaira-type vanishing results for (log) Fano varieties (as an instance of this principle see the main result of \cite{HW17}).    \\
\indent If one is interested in developing the MMP for threefolds in characteristic $p \leq 5$, it would be important to know an optimal bound for $p_0(3)$. Since we know by \cite[Theorem 1.1]{CT16} that there exists a threefold plt pair with non-normal centre in characteristic two we ask the following
\begin{question}
Are plt threefold centres normal in characteristic three and five?
\end{question} 
In \cite[Theorem 2]{Ber17} we constructed an example of a klt not Cohen-Macaulay threefold singularity in characteristic three. This suggests that more pathologies for threefold singularities could appear in characteristic three.
\bibliographystyle{amsalpha}
\bibliography{bibliography}
\Addresses
\end{document}